\documentclass[reqno]{amsart}
\usepackage{amsfonts}
\usepackage{graphicx}
\usepackage{amscd}
\usepackage{amsmath}
\usepackage{amssymb}
\usepackage{latexsym}
\usepackage{hyperref}
\usepackage{mathrsfs}
\usepackage[all]{xy}
\usepackage[latin1]{inputenc}
\usepackage[T1]{fontenc}
\theoremstyle{plain}

\newtheorem{definition}{\bf Definition}
\newtheorem{example}{\bf Example}
\newtheorem*{example.a}{\bf Example A.}
\newtheorem{lemma}{\bf Lemma}

\newtheorem{proposition}{\bf Proposition}
\newtheorem{remark}{Remark}

\newtheorem{theorem}{\bf Theorem}
\newtheorem*{theorem.a}{\bf Theorem A}

\numberwithin{equation}{section}
\newcommand\tr{\mathrm{tr}}
\newcommand\dv{\mathrm{div}}
\begin{document}

\title[A characterization of quasi-Einstein metrics]{A characterization of quasi-Einstein metrics}

\author{Antonio Airton Freitas Filho}
\address{Departamento de Matem\'atica, Universidade Federal do Amazonas, Av. Rodrigo Oct\'avio, 6200, 69080-900 Manaus, Amazonas, Brazil}
\email{aafreitasfilho@ufam.edu.br}
\urladdr{http://www.ufam.edu.br}




\date{October 14, 2025}

\begin{abstract}
We study the modified Ricci solitons as a new class of Einstein type metrics that contains both Ricci solitons and $n$-quasi-Einstein metrics. This class is closely related to the construction of the Ricci solitons that are realised as warped products. A modified Ricci soliton appears as part of a special solution of the modified Ricci-harmonic flow, which result a new characterization of $n$-quasi-Einstein metrics. We also study the modified Ricci almost solitons. In the spirit of the Lichnerowicz and Obata first eigenvalue theorems, we prove that in the class of compact Riemannian manifolds with constant scalar curvature the standard sphere with a structure of gradient modified Ricci almost soliton is rigid under some specific geometric conditions. Moreover, we display an example of modified Ricci-harmonic soliton.

\end{abstract}
\maketitle


\section{Introduction}

It this paper, we give an alternative characterization of the $n$--quasi-Einstein manifolds. These latter are complete Riemannian manifolds whose corresponding modified Bakry-\'Emery Ricci tensor is a constant multiple of the metric tensor. They are base of Einstein warped products with $n$--dimensional Einstein fibers \cite{besse}. Here, we discuss an interesting case of Einstein type manifolds which stems from a study of Ricci soliton warped product. For instance, Bryant constructed a steady Ricci soliton as the warped product $(0,+\infty) \times_f \mathbb{S}^{n}$, $n>1$, with a radial warping function $f$. Bryant did not himself publish this result, but it can be checked in \cite{Chow}. Pina and Sousa~\cite{romildo} constructed steady gradient Ricci solitons warped product with base conformal to an $m$--dimensional pseudo Euclidean space invariant under the action of an $(m-1)$--dimensional translation group ($m\geqslant3$). Noncompact gradient steady Ricci solitons were constructed by Ivey~\cite{Ivey2} and by Dancer and Wang~\cite{DW}, which were achieved by the use of doubly and multiple warped product constructions. Gastel and Kronz~\cite{GK} constructed a two-parameter family (doubly warped product) of gradient expanding Ricci solitons on $\mathbb{R}^{m}\times F^n$, where $F^{n}$ $(n\geqslant2)$ is an Einstein manifold with positive scalar curvature. Complementing these works, Feitosa et al.~\cite{RSWP} established the necessary and sufficient conditions for constructing a gradient Ricci soliton warped product. As an application, they presented a example of expanding Ricci soliton warped product having as fiber an Einstein manifold with non-positive scalar curvature. We note that this example is rigid in the sense of Petersen and Wylie~\cite{PW-2009}. Motivated by these works, we consider the following class of metrics.


A modified Ricci soliton is a complete Riemannian manifold $(M^m,g)$ endowed with a vector field $X:M \to TM$ and a $1$--form $\omega:M \to TM^*$ satisfying
\begin{equation}\label{eq.tSR}
Rc_{g}+\dfrac{1}{2}\mathcal{L}_{X}g=\lambda g+\theta\omega\otimes\omega,
\end{equation}
for some two real numbers $\lambda$ and $\theta>0.$ Here, $\mathcal{L}_{X}g$ stands for the Lie derivative of $g$ with respect to $X$, and $Rc_{g}$ for the Ricci tensor of $g$.

For simplicity, we say that $(M^m,g,X,\omega)$ is a modified Ricci soliton. The condition $\theta>0$ is just a motivational issue, but there is no impediment to consider $\theta\leqslant0$. When $X=\nabla\eta$ and $\omega=d\xi$ for some smooth functions $\eta,\,\xi$ on $M$, then $(M^m,g,\nabla\eta,d\xi)$ is a gradient modified Ricci soliton. Thus, Eq.~\eqref{eq.tSR} becomes
\begin{equation}\label{eq.tSRgr}
Rc_{g}+\nabla d\eta=\lambda g+\theta d\xi\otimes d\xi,
\end{equation}
where $\nabla d\eta$ is the hessian of $\eta$ in the metric $g.$

Gradient modified Ricci solitons appeared previously in the work of Cho and Kimura~\cite{cho-kimura} in the studying of real hypersurfaces in nonflat complex space forms. In this case, the authors used the nomenclature of $\theta$--Ricci soliton.

When $\lambda$ is a nonconstant smooth function on $M$, we refer to $(M^m,g,X,\omega)$ as modified Ricci almost soliton.

In this paper, we prove that in the class of compact Riemannian manifolds with constant scalar curvature, the standard sphere with a structure of gradient modified Ricci almost soliton is rigid provided some specific geometric conditions. This is the subject of our main result which is in the spirit of the Lichnerowicz and Obata first eigenvalue theorems and is described as follows.

\begin{theorem}\label{rig.cpct}
Let $(M^m,g,X,d\xi)$, $m\geqslant2,$ be a compact modified Ricci almost soliton with constant scalar curvature $S$. If $\xi$ is an eigenfunction of Laplacian, then its correspondent eigenvalue $\alpha$ satisfies
\begin{equation*}
\alpha\geqslant\dfrac{S}{m-1}.
\end{equation*}
The equality occurs only if $(M^m,g)$ is isometric to a standard sphere $\mathbb{S}^{m}(r)$, where $r=\sqrt{m(m-1)/S}$. Moreover, up to rescaling and a constant, the structure is gradient with functions given in Proposition \ref{ex.p.Mc}.
\end{theorem}

\section{Motivating the definition of the modified Ricci soliton}\label{ttRF}

The Ricci solitons are the corresponds special (self-similar) solutions of the Ricci flow which was introduced by Hamilton in \cite{hamilton1}. They has been extensively studied by many researchers as generalizations of Einstein manifolds. A Ricci soliton $(M^m,g,X,\lambda)$ is a Riemannian manifold $(M^m,g)$ endowed with a vector field $X\in\mathfrak{X}(M)$ satisfying
\begin{equation}
Rc_{g}+\dfrac{1}{2}\mathcal{L}_{X}g=\lambda g,
\end{equation}
for some constant $\lambda$. A Ricci soliton is expanding, steady or shrinking if $\lambda<0$, $\lambda=0$ or $\lambda>0$, respectively.  A gradient Ricci soliton satisfies
\begin{equation}
Rc_{g}+\nabla d\psi=\lambda g
\end{equation}
for some smooth potential function $\psi$ on $M$.

Studying gradient Ricci solitons $(\bar{M}^{m+n},\bar{g},\nabla\tilde\varphi,\lambda)$ that are realised as warped products $(\bar{M}^{m+n},\bar{g})= M^m\times_fF^n$, we observe that the base $M$ represents a generalization of $n$--quasi-Einstein metric, where $\tilde\varphi$ is the lift of $\varphi\in C^\infty(M)$ to $\bar{M}$ and $f$ is the warping function. Indeed, there exists a fundamental property of the functions $\varphi$ and $f$, namely:
\begin{equation}\label{EQMthmIntr}
Rc_{g}+\nabla d\varphi=\lambda g+\dfrac{n}{f}\nabla df,
\end{equation}
which consists in a Ricci-Hessian type equation, and can be found in \cite{RSWP}. We point out that the class of Ricci-Hessian type equations contain the class of $n$--quasi-Einstein metrics. For instance, for an $r$--quasi-Einstein metric $(M^m,g,\nabla h,\lambda)$ satisfying
\begin{equation}\label{r-QEM}
Rc_{g}+\nabla dh-\dfrac{1}{r}dh\otimes dh=\lambda g,
\end{equation}
we can take $r=4n$, $\varphi=\dfrac{h}{2}$ and $f=e^{-\varphi/n}$ in order to verify that $(M^m,g,\varphi,f,\lambda)$ satisfies the Eq.~\eqref{EQMthmIntr}. But, if we assume that $(M^m,g,\varphi,f,\lambda)$ satisfies an equation of the type \eqref{EQMthmIntr}. Then, we have
\begin{equation}\label{Eq3AuxEx1}
Rc_{g}+\nabla d\eta  = \lambda g + \dfrac{1}{n}d\xi\otimes d\xi,
\end{equation}
where $\xi=-n\ln(f)$ and $\eta=\varphi+\xi$. Note that if $\nabla\varphi$ is a homothetic vector field, then there is no distinction between the Eqs.~\eqref{EQMthmIntr} and \eqref{r-QEM}. However, it is well-known that the existence of a nontrivial homothetic vector field imply restrictions on the geometry of the manifold.

To complement our motivation, we assume that $(\bar{M},\bar{g},\tilde{Y},\lambda)$ is a Ricci soliton warped product, where $\bar{M}^{m+n}=M^m\times_fF^n$ and $\tilde Y$ is the horizontal lift of $Y\in \mathfrak{X}(M)$ to $\mathfrak{X}(\bar{M})$. Proposition $35$ in O'Neill~\cite{oneill} ensures that
\begin{equation*}
\mathcal{L}_{\tilde{Y}}\bar{g}=\widetilde{\mathcal{L}_Yg}.
\end{equation*}
By Corollary $43$ of \cite{oneill} the following equation is verified on the base
\begin{equation}\label{BX}
Rc_{g}+\dfrac{1}{2}\mathcal{L}_Yg=\lambda g+\dfrac{n}{f}\nabla df.
\end{equation}
Taking $\xi=-n\ln(f)$, we get
\begin{equation*}
\dfrac{n}{f}\nabla df=-\nabla d\xi + \dfrac{1}{n}d\xi\otimes d\xi.
\end{equation*}
Hence, Eq.~\eqref{BX} becomes
\begin{equation*}
Rc_{g}+\dfrac{1}{2}\mathcal{L}_Xg=\lambda g+\dfrac{1}{n}d\xi\otimes d\xi,
\end{equation*}
where $X=Y+\nabla\xi$. This latter equation was the first motivation to consider the subject of study this paper. In Section~\ref{Sec-CMRS}, we give the second motivation which is inspired by the Ricci-harmonic flow introduced by M\"uller \cite{muller}.

\begin{example}\label{Berger}
Let us consider $\mathbb{S}^3 =\{(z,w)\in \mathbb{C}^{2};\,|z|^{2}+|w|^{2}=1\}$ endowed with the family of metrics
\begin{equation*}
g_{\kappa,\tau}=\dfrac{4}{\kappa}\Big[g_\circ+ \left(\dfrac{4\tau^{2}}{\kappa}-1 \right) V^{\flat}\otimes V^{\flat}\Big],
\end{equation*}
where $g_\circ$ stands for the Euclidean metric on $\mathbb{S}^3$, $\,V_{(z,w)}=(iz,iw)$ for each $(z,w)\in \mathbb{S}^3$ while $\kappa,\,\tau$ are real numbers with $\kappa>0$ and $\tau\neq0.$  A Berger sphere is $(\mathbb{S}^3,g_{\kappa,\tau})$, which we denote  by $\mathbb{S}^3_{\kappa,\tau}$. Next we choose  $E_3=\dfrac{\kappa}{4\tau}V,$ in order to deduce that the Ricci tensor of $\mathbb{S}^3_{\kappa,\tau}$ satisfies
\begin{equation*}
Rc_{g_{\kappa,\tau}} = (\kappa-2\tau^2)g_{\kappa,\tau} +(4\tau^2-\kappa)E_3^{\flat}\otimes E_3^{\flat}.
\end{equation*}
Taking into account that $E_3$ is a Killing field, if we define $X = E_3$ as well as $\lambda=\kappa-2\tau^2$, then  we arrive at
\begin{equation*}
Rc_{g_{\kappa,\tau}}+\dfrac{1}{2}\mathcal{L}_{X}g_{\kappa,\tau} = \lambda g_{\kappa,\tau}+X^{\flat}\otimes X^{\flat}.
\end{equation*}
According to \eqref{eq.tSR}, $g_{\kappa,\tau}$ is  a modified Ricci soliton on $\mathbb{S}^3_{\kappa,\tau}$.
\end{example}

\section{Rigidity of compact modified Ricci almost soliton}

We begin by noting that for any $u\in C^\infty(M)$ holds
\begin{equation*}
du\otimes du=\dfrac{1}{2}\nabla du^2-u\nabla^2u,
\end{equation*}
so that Eq.~\eqref{eq.tSRgr} is equivalent to
\begin{equation}\label{f.grad.equ.}
Rc_{g}+\nabla d\Big(\eta-\dfrac{\theta}{2}\xi^2\Big)=\lambda g-\theta\xi\nabla d\xi.
\end{equation}

Eq.~\eqref{f.grad.equ.} will be useful to characterize the structures of modified Ricci almost solitons on the unit standard sphere $\mathbb{S}^m$. For this, we denote by $g$ the canonical metric on $\mathbb{S}^m$ and we consider the Obata's solution for the equation $\nabla d\xi =\dfrac{\Delta\xi}{m}g$. More precisely, we must have $\Delta (\Delta \xi)+m\Delta\xi=0$, see Obata~\cite{obata}. Whence, there exists a vector $v\in\mathbb{S}^m$ such that $\Delta\xi=h_{v} =-\dfrac{1}{m} \Delta h_{v}$, where $h_v(x)=\langle x,v\rangle$ is a height function on $\mathbb{S}^m$. Thus, for some constant $c_1$ we have
\begin{equation}\label{func xi}
\xi=c_1-\dfrac{h_v}{m}.
\end{equation}
Since $Rc_{g}=(m-1)g$ and $-m\nabla d\xi=\nabla dh_v=-h_v g$ we replace \eqref{func xi} into \eqref{f.grad.equ.} to get
\begin{equation}\label{nabla.eta.xi2}
\nabla d\Big[\eta-\dfrac{\theta}{2}\Big(c_1-\dfrac{h_v}{m}\Big)^2\Big]=\Big[\lambda+\theta\Big(\dfrac{h_v}{m}-c_1\Big)\dfrac{h_v}{m}-(m-1)\Big]g.
\end{equation}
Therefore, $\nabla \Psi$ is a conformal vector field on $(\mathbb{S}^m,g)$, where $\Psi=\eta-\dfrac{\theta}{2}\Big(c_1-\dfrac{h_v}{m}\Big)^2$. Again there must exist a vector $w\in\mathbb{S}^m$ and a constant $c_2$ such that
\begin{equation}\label{func eta}
\eta=\dfrac{\theta}{2}\Big(c_1-\dfrac{h_v}{m}\Big)^2-\dfrac{h_w}{m}+c_2.
\end{equation}
From \eqref{nabla.eta.xi2} and \eqref{func eta} we have
\begin{equation*}
-h_wg=\nabla d h_w=-m\nabla d\Psi=-m\Big[\lambda+\theta\Big(\dfrac{h_v}{m}-c_1\Big)\dfrac{h_v}{m}-(m-1)\Big]g.
\end{equation*}
Then
\begin{equation*}
\dfrac{h_w}{m}=\lambda+\theta\Big(\dfrac{h_v}{m}-c_1\Big)\dfrac{h_v}{m}-(m-1),
\end{equation*}
that is,
\begin{equation}\label{func lamb}
\lambda=\theta\Big(c_1-\dfrac{h_v}{m}\Big)\dfrac{h_v}{m}+\dfrac{h_w}{m}+m-1.
\end{equation}

With this considerations in mind we can write our first result as follows.

\begin{proposition}\label{ex.p.Mc}
The unique structures of gradient modified Ricci almost soliton $(\mathbb{S}^m,g,\nabla\eta,d\xi)$ with $\nabla\xi$ being a conformal vector field are described by smooth functions in \eqref{func xi}, \eqref{func eta} and \eqref{func lamb}.
\end{proposition}

For the next results we recall that the traceless part of a $(0,2)$--tensor $T$ in $(M^m,g)$ is defined by
\begin{equation*}
\mathring{T}=T-\dfrac{\tr(T)}{m}g
\end{equation*}
and the divergence of its corresponding $(1,1)$--tensor $T$ is defined as the $(0,1)$--tensor
\begin{equation*}
(\dv T)(v)(p) = \mathrm{tr}\big(w \mapsto (\nabla_w T)(v)(p)\big),
\end{equation*}
where $p\in M$, $v,w\in T_pM,$ $\nabla$ stands for the covariant derivative of $T$ and $\mathrm{tr}$ is the trace calculated in the metric $g.$

In what follows, we consider the musical isomorphism $\,^{\sharp}\!:\mathfrak{X}^*(M)\to \mathfrak{X}(M)$, that is, the inverse map of $\,^{*}\!:\mathfrak{X}(M)\to \mathfrak{X}^*(M)$ given by $X^*(\cdot)=g(X,\cdot)$, for all $X\in \mathfrak{X}(M)$.

\begin{lemma}\label{lem.int-cpct0}
Let $(M^m,g,X,\omega)$ be a compact modified Ricci almost soliton. Then holds the following identity
\begin{equation}\label{eq.int-cpct0}
\int_M\|\mathring{Rc_{g}}\|^2dM=\dfrac{m-2}{2m}\int_{M}\langle\nabla S,X\rangle dM+\theta\int_M\mathring{Rc_{g}}(\omega^\sharp,\omega^\sharp)dM.
\end{equation}
In particular, for $\omega=d\xi$ one has
\begin{align}\label{eq.int-cpct}
\nonumber\int_M\|\mathring{Rc_{g}}\|^2dM=&\ \dfrac{m-2}{2m}\int_{M}\langle\nabla S,X\rangle dM-\theta\int_M\|\mathring{\nabla d\xi}\|^2dM\\
& +\theta\dfrac{m-1}{m}\int_M\Big[(\Delta\xi)^2-\dfrac{S}{m-1}\|\nabla\xi\|^2\Big]dM.
\end{align}
\end{lemma}
\begin{proof} A straightforward computation shows that
\begin{align*}
\dv(\mathring{Rc_{g}}(X)) &= (\dv\mathring{Rc_{g}})(X)+\langle\mathring{Rc_{g}},\nabla X\rangle\\
&= \dfrac{m-2}{2m}\langle\nabla S,X\rangle+\langle\mathring{Rc_{g}},\dfrac{1}{2}\mathcal{L}_Xg\rangle\\
&= \dfrac{m-2}{2m}\langle\nabla S,X\rangle+\langle\mathring{Rc_{g}},-Rc_{g}+\lambda g+\theta\omega\otimes\omega\rangle.
\end{align*}
Therefore,
\begin{equation}\label{general}
\|\mathring{Rc_{g}}\|^2=-\dv(\mathring{Rc_{g}}(X))+\dfrac{m-2}{2m}\langle\nabla S,X\rangle+\theta\mathring{Rc_{g}}(\omega^\sharp,\omega^\sharp).
\end{equation}
We obtain \eqref{eq.int-cpct0} by integrating of \eqref{general}. Now, we use the Bochner's formula to get
\begin{align*}
\mathring{Rc_{g}}(\nabla u,\nabla u) &= Rc_{g}(\nabla u,\nabla u)-\dfrac{S}{m}\|\nabla u\|^2\\
&=\dfrac{1}{2}\Delta\|\nabla u\|^2-\|\nabla d u\|^2-\langle\nabla\Delta u,\nabla u\rangle-\dfrac{S}{m}\|\nabla u\|^2\\
&=\dfrac{1}{2}\Delta\|\nabla u\|^2-\|\mathring{\nabla du}\|^2-\dfrac{1}{m}(\Delta u)^2-\langle\nabla\Delta u,\nabla u\rangle-\dfrac{S}{m}\|\nabla u\|^2.
\end{align*}
Applying Stokes' theorem, we obtain
\begin{align}\label{CStokes}
\int_M\Big[\mathring{Rc_{g}}(\nabla u,\nabla u) +\|\mathring{\nabla d u}\|^2\Big] dM =\dfrac{m-1}{m}\int_M\Big[(\Delta u)^2-\dfrac{S}{m-1}\|\nabla u\|^2\Big]dM
\end{align}
for all $u\in C^\infty(M)$. In particular, taking $u=\xi$ in \eqref{CStokes} and $\omega=d\xi$ in \eqref{eq.int-cpct0} we complete the proof.
\end{proof}

\subsection{Proof of Theorem \ref{rig.cpct}}
\begin{proof}
In the hypothesis of theorem we have from \eqref{eq.int-cpct}
\begin{equation*}
0\leqslant\int_M\|\mathring{Rc_{g}}\|^2dM + \theta\int_M\|\mathring{\nabla d\xi}\|^2dM=\theta\dfrac{m-1}{m}\Big(\alpha-\dfrac{S}{m-1}\Big)\int_M\|\nabla\xi\|^2dM.
\end{equation*}
So, the eigenvalue $\alpha$ must satisfy $\alpha\geqslant\dfrac{S}{m-1}$. If the equality occurs, then
\begin{equation*}
\nabla d\xi=-\dfrac{S\xi}{m(m-1)}g.
\end{equation*}
Thus, by a classical result due to Obata \cite{obata} we have that $(M^m,g)$ is isometric to a standard sphere $\mathbb{S}^m(r)$, where $r=\sqrt{m(m-1)/S}$.
It remains to prove that $X$ can be changed by $\nabla\eta$ for some smooth function $\eta$. By Eq.~\eqref{eq.tSR} one has
\begin{align*}
\dfrac{S}{m}g+\dfrac{1}{2}\mathcal{L}_Xg &= \lambda g+\theta d\xi\otimes d\xi\\
&= \lambda g- \theta\xi\nabla d\xi+\dfrac{\theta}{2}\nabla d\xi^2\\
&= \lambda g+\dfrac{\theta S\xi^2}{m(m-1)}g+\dfrac{\theta}{2}\nabla d\xi^2.
\end{align*}
Whence
\begin{equation*}
\mathcal{L}_Zg=2\rho g,
\end{equation*}
where
\begin{equation*}
Z=X-\dfrac{\theta}{2}\nabla\xi^2 \ \ \mbox{and} \ \
\rho=\dfrac{\dv Z}{m}=\lambda+\dfrac{\theta S\xi^2}{m(m-1)}-\dfrac{S}{m}.
\end{equation*}
Since $(\mathbb{S}^m(r),g)$ is Einstein manifold, we conclude that the conformal factor $\rho$ is a eigenfunction of the Laplacian with eigenvalue $\dfrac{S}{m-1}$. Then, from \eqref{CStokes} we obtain the Obata's equation for $\rho$
\begin{equation*}
\nabla d\rho=-\dfrac{S\rho}{m(m-1)}g,
\end{equation*}
so setting
\begin{equation*}
u=-\dfrac{m(m-1)}{S}\rho,
\end{equation*}
we get
\begin{align*}
\dfrac{1}{2}\mathcal{L}_{\nabla u}g&=\nabla du=-\dfrac{m(m-1)}{S}\nabla d\rho
=\rho g\\
&=\dfrac{1}{2}\mathcal{L}_Zg=\dfrac{1}{2}\mathcal{L}_Xg-\dfrac{1}{2}\mathcal{L}_{\frac{\theta}{2}\nabla\xi^2}g.
\end{align*}
Thus
\begin{align*}
\dfrac{1}{2}\mathcal{L}_Xg
&=\dfrac{1}{2}\mathcal{L}_{\nabla\left(u+\frac{\theta}{2}\xi^2\right)}=:\frac{1}{2}\mathcal{L}_{\nabla\eta}g.
\end{align*}
Hence, up to rescaling and a constant, $(\mathbb{S}^m(r),g,\nabla\eta,d\xi)$ is a gradient modified Ricci almost soliton with functions given by Proposition~\ref{ex.p.Mc}.
\end{proof}

\section{An obstruction in the case compact and a compactness criterion.}

We now present two results to modified Ricci almost solitons: an obstruction in the case compact and a compactness criterion.

\begin{proposition}
Let $(M^m,g,X,\omega)$ be a compact modified Ricci almost soliton. If $m\lambda-S$ is a nonnegative function, then $\omega$ is identically null and $X$ is a Killing vector field.
\end{proposition}
\begin{proof}
Initially by Hodge-de Rham decomposition theorem (see e.g. \cite{warner} p. 223) we can write $X=\nabla h +Y$, where $h\in C^\infty(M)$ and $Y\in\mathfrak{X}(M)$ with $\dv Y=0$. Now, taking the trace in Eq.~\eqref{eq.tSR} we obtain
\begin{equation*}
\Delta h=m\lambda-S + \theta\|\omega\|^2\geqslant 0,
\end{equation*}
i.e., $h$ is a subharmonic function, so $h$ is constant by Hopf's theorem. Then,
\begin{equation*}
0=\Delta h= m\lambda-S +\theta\|\omega\|^2\geqslant0.
\end{equation*}
Since both $m\lambda-S$ and $\theta\|\omega\|^2$ are nonnegative, follows that $m\lambda=S$ and $\omega\equiv0$. Moreover, $X=Y$ and
\begin{equation}\label{E-K}
\mathring{Rc_{g}}=-\dfrac{1}{2}\mathcal{L}_Yg.
\end{equation}
Using Eq.~\eqref{eq.int-cpct0} we deduce, for $m\geqslant3,$ that
\begin{equation*}
0\leqslant\int_M\|\mathring{Rc_{g}}\|^2dM=\dfrac{m-2}{2m}\int_M\langle\nabla S,Y\rangle dM=-\dfrac{m-2}{2m}\int_M S\dv YdM=0.
\end{equation*}
Hence, $(M^m,g)$ is an Einstein manifold and $Y$ is a Killing vector field by \eqref{E-K}. For $m=2$, the conclusion of $Y$ to be a Killing vector field follows from Eq.~\eqref{E-K}, since every surface is an Einstein manifold.
\end{proof}

\begin{proposition}\label{cpct.2}
Let $(M^m,g)$ be a complete Riemannian manifold satisfying
\begin{equation*}
Rc_{g}+\dfrac{1}{2}\mathcal{L}_Xg\geqslant\lambda g+\theta\omega\otimes\omega
\end{equation*}
for a vector field $X:M \to TM$, a $1$--form $\omega:M \to TM^*$, a smooth function $\lambda\in C^\infty(M)$ and a constant $\theta$. Then, $M$ is compact provided that there exist positive constants $c_1$ and $c_2$ such that: $\|X\|\leqslant c_1$ and $(\lambda\circ\gamma)+\theta \big(\omega(\gamma')\big)^2\geqslant c_2$ for all geodesic $\gamma$. Moreover,
\begin{equation*}
\mathrm{diam}(M)\leqslant\dfrac{\pi}{c_2}\left(c_1+\sqrt{c_1^2+(m-1)c_2}\right).
\end{equation*}
In particular, the fundamental group $\pi_1(M)$ is finite.
\end{proposition}
\begin{proof}
By hypothesis we have
\begin{equation*}
Rc_{g}\geqslant\lambda g +\theta\omega\otimes\omega -\dfrac{1}{2}\mathcal{L}_Xg.
\end{equation*}	
Let $p$ be a point in $M$ and consider any geodesic $\gamma:[0,\infty)\rightarrow M$ emanating from $p$ and parameterized by arc length $s.$ Thus,
\begin{align*}
Rc_{g}(\gamma',\gamma') &\geqslant (\lambda\circ\gamma)\cdot g(\gamma',\gamma') +\theta\big(\omega(\gamma')\big)^2 -\dfrac{1}{2}(\mathcal{L}_Xg)(\gamma',\gamma')\\
&= (\lambda\circ\gamma) +\theta\big(\omega(\gamma')\big)^2 -\dfrac{d}{ds}g(X,\gamma')\\
&\geqslant c_2 +\dfrac{d}{ds}g(-X,\gamma').
\end{align*}
Then, defining the function $f:[0,\infty)\rightarrow\mathbb{R}$ by $f(s)=g(-X_{\gamma(s)},\gamma'(s))$ we obtain by Cauchy-Schwarz inequality that
\begin{equation*}
\|f(s)\|\leqslant\|X_{\gamma(s)}\|\leqslant c_1.
\end{equation*}
Hence, we can use Theorem $1.2$ of Galloway~\cite{galloway}, from which $M$ is compact and satisfy the required diameter estimate. The fundamental group $\pi_1(M)$ is finite by \cite[Theorem $1.3$]{galloway}.
\end{proof}

\begin{remark}
We point out that if $\theta\geqslant0$ and $\lambda$ is a positive constant, the previous result is covered by M. Fern\'andez-L\'opez and E. Garc\'ia-R\'io~\cite[Theorems~$1$ and $2$]{FLGR} in which the authors used the Ambrose's compactness criterion~\cite{Ambrose}. In addition, we can again apply the Theorem $1.2$ in \cite{galloway} for the following estimated diameter:
\begin{equation*}
\mathrm{diam}(M)\leqslant\dfrac{\pi}{\lambda}\left(c_{1}+\sqrt{c_{1}^{2}+(m-1)\lambda}\right).
\end{equation*}
\end{remark}

For instance, we recall that Berger's Sphere $\mathbb{S}^{3}_{\kappa,\tau}$, in Example \ref{Berger}, with $4\tau^2-\kappa>0$, we have the estimate
\begin{align*}
Rc_{g}(\gamma',\gamma')&=(\kappa-2\tau^2)\cdot g_{\kappa,\tau}(\gamma',\gamma')+(4\tau^2-\kappa)\cdot\big(E_3^*(\gamma')\big)^2\\
&=(\kappa-2\tau^2)+(4\tau^2-\kappa)\cdot\big(E_3^*(\gamma')\big)^2\\
&\geqslant\kappa-2\tau^2,
\end{align*}
for all geodesics $\gamma$. If $\kappa>2\tau^2$, taking $X=E_3$, $\lambda=\kappa-2\tau^2,$ $\omega=E_3^*$ and $\theta=4\tau^2-\kappa$ we can verify easily that the conditions in Proposition \ref{cpct.2} are satisfied.


\section{Characterizing the modified Ricci soliton}\label{Sec-CMRS}

In this section, we characterize a modified Ricci soliton as part of a special solution of the modified Ricci-harmonic flow.

\subsection{Short-time existence and uniqueness\label{Exist-Unicit}}

In this section we introduce the modified Ricci-harmonic flow, then we prove, for the compact case, the short-time existence and uniqueness theorem for this flow. For this we will use the ``DeTurck's Trick'' (see \cite{DeT}) and script written by Brendle in Section $2.2$ his book \cite{simon}, which refers to the short-time existence and uniqueness for the flow of Ricci (the proof present here is {\it mutatis mutandis} of this script). For our case, we define the modified Ricci-harmonic-DeTurck flow and finally we define the modified Ricci-harmonic solitons as special solutions of the modified Ricci-harmonic flow.

\begin{definition}\label{tFR}
Let $M^m,N^n$ be smooth manifolds and $h$ be a fixed Riemannian metric and $\omega_N$ be a $1$--form on $N.$ Let $g(t)$ be a one-parameter family of Riemannian metrics on $M$ and $\phi(t)$ be a one-parameter family of smooth maps from $M$ to $N$ and consider $\omega(t):=\phi(t)^*\omega_N$. We say that $(g(t),\phi(t))$, $t\in[0,\varepsilon)$, is a solution to the modified Ricci-harmonic flow
\begin{equation}\label{eq.tFR}
\begin{cases}
\dfrac{\partial}{\partial t}g(t)=-2Rc_{g(t)}+2\theta\omega(t)\otimes\omega(t),\\
\dfrac{\partial}{\partial t}\phi(t)=\Delta_{g(t),h}\phi(t),
\end{cases}
\end{equation}
if it satisfies \eqref{eq.tFR}, where $\theta$ is a positive constant and $\Delta_{g(t),h}\phi(t)$ is the tension field of map $\phi(t)$ with respect to metrics $g(t)$ and $h$, under evolution of metrics $g(t)$ and maps $\phi(t)$ for each $t\in[0,\varepsilon)$ varying smoothly.
\end{definition}

We also note that the modified Ricci-harmonic flow generalizes the flow introduced by List in \cite{list} and is a modification of the Ricci-harmonic flow defined by M\"uller in \cite{muller}, for this reason we adopt such a nomenclature.

\begin{definition}\label{tFR-DeT}
Let $(M^m,\bar{g})$ and $(N^n,h)$ be two fixed Riemannian manifolds, and $\omega_N$ a fixed $1$--form on $N$. Let $\tilde{g}(t)$ be a one-parameter family of Riemannian metrics on $M$,  $\tilde{\phi}(t)$ a one-parameter family of smooth maps from $M$ to $N$ and consider $\tilde{\omega}(t):=\tilde{\phi}(t)^*\omega_N$. We say that $(\tilde{g}(t),\tilde{\phi}(t))$, $t\in[0,\varepsilon)$, is a solution to the modified Ricci-harmonic-DeTurck flow
\begin{equation}\label{FRDTm}
\begin{cases}
\dfrac{\partial}{\partial t}\tilde{g}(t)=-2Rc_{\tilde{g}(t)}+2\theta\tilde{\omega}(t)\otimes\tilde{\omega}(t)-\mathcal{L}_{Z_t}\tilde{g}(t),\\
\dfrac{\partial}{\partial t}\tilde{\phi}(t)=\Delta_{\tilde{g}(t),h}\tilde{\phi}(t)-\mathcal{L}_{Z_t}\tilde{\phi}(t),
\end{cases}
\end{equation}
it is satisfies \eqref{FRDTm}, where $\theta$ is a positive constant, $\Delta_{\tilde{g}(t),h}\tilde{\phi}(t)$ is the tension field of map $\tilde{\phi}(t)$ with respect to metrics $\tilde{g}(t)$ and $h$, under evolution of metrics $\tilde{g}(t)$ and maps $\tilde{\phi}(t)$ for each $t\in[0,\varepsilon)$ varying smoothly. While $Z_t=\Delta_{\tilde{g}(t),\bar{g}}\mathrm{id}_M$ is the tension field of identity map $\mathrm{id}_M:M\to M$ with respect to metrics  $\tilde{g}(t)$ e $\bar{g}$, under evolution of metrics $\tilde{g}(t)$ for each $t\in[0,\varepsilon)$ varying smoothly.
\end{definition}

The result below assures the short-time existence and uniqueness for the modified Ricci-harmonic-DeTurck flow.

\begin{proposition}\label{exs.tFRDet}
Let $(M^m,\bar{g})$ and $(N^n,h)$ be two fixed compact Riemannian manifolds, and $\omega_N$ a fixed $1$--form on $N$. Given a Riemannian metric $g$ em $M$ and a smooth map $\phi:(M,g)\to(N,h)$, there exist a real number $\varepsilon>0,$ one-parameter families of metrics
$\tilde{g}(t)$, and smooth maps $\tilde{\phi}(t)$, $t\in[0,\varepsilon)$, such that  $(\tilde{g}(t),\tilde{\phi}(t))$, $t\in[0,\varepsilon)$, is a solution of modified Ricci-harmonic-DeTurck flow, with initial data $(\tilde{g}(0),\tilde{\phi}(0))=(g,\phi)$. Moreover, the solution  $(\tilde{g}(t),\tilde{\phi}(t))$, $t\in[0,\varepsilon)$, is unique.
\end{proposition}
\begin{proof} Let $(x_i)$ and $(y_\gamma)$ local coordinate systems for smooth manifolds $M^m$ and $N^n$, respectively. For simplicity, we will omit the time variable $t$. Now, we write the involved objects of modified Ricci-harmonic-DeTurck flow in local coordinates. We begin by Ricci tensor in the metric $\tilde{g}$.
\begin{align*}
Rc_{\tilde{g}}&=-\dfrac{1}{2}\sum_{i,j,k,l=1}^{m}\tilde{g}^{ij}\left(\dfrac{\partial^2\tilde{g}_{kl}}{\partial x_i \partial x_j}-\dfrac{\partial^2\tilde{g}_{kj}}{\partial x_i \partial x_l}-\dfrac{\partial^2\tilde{g}_{il}}{\partial x_k \partial x_j}+\dfrac{\partial^2\tilde{g}_{ij}}{\partial x_k \partial x_l}\right)dx_k\otimes dx_l\\
&\ \ \ \ + \ \mbox{lower order terms},
\end{align*}
where  "lower order terms" refers to order terms $0$ and $1$ with respect to partial derivatives of $\tilde{g}_{kl}$ and, after, with respect to partial derivatives of $\phi_\gamma$.
Moreover, we define the vector field $Z=\Delta_{\tilde{g},\bar{g}}\mathrm{id}_M$, i.e.,
\begin{equation*}
Z=\sum_{i,j,l=1}^{m}\tilde{g}^{ij}\left(\bar{\Gamma}_{ij}^l-\tilde{\Gamma}_{ij}^l\right)\partial_l,
\end{equation*}
where $\bar{\Gamma}$ e $\tilde{\Gamma}$ denote the Christoffel symbols associated with metrics $\bar{g}$ and $\tilde{g}$, respectively. Implying
\begin{align*}
Z &= -\dfrac{1}{2}\sum_{i,j,k,l=1}^{m}\tilde{g}^{ij}\tilde{g}^{kl}\left(\dfrac{\partial\tilde{g}_{kj}}{\partial x_i}+\dfrac{\partial\tilde{g}_{ik}}{\partial x_j}-\dfrac{\partial\tilde{g}_{ij}}{\partial x_k}\right)\partial_l\\
&\ \ \ \ +\dfrac{1}{2}\sum_{i,j,k,l=1}^{m}\tilde{g}^{ij}\bar{g}^{kl}\left(\dfrac{\partial\bar{g}_{kj}}{\partial x_i}+\dfrac{\partial\bar{g}_{ik}}{\partial x_j}-\dfrac{\partial\bar{g}_{ij}}{\partial x_k}\right)\partial_l.
\end{align*}
From this we deduce that
\begin{align*}
\mathcal{L}_Z\tilde{g}&=-\sum_{i,j,k,l=1}^{m}\tilde{g}^{ij}\left(\dfrac{\partial^2\tilde{g}_{kj}}{\partial x_i \partial x_l}+\dfrac{\partial^2\tilde{g}_{il}}{\partial x_k \partial x_j}-\dfrac{\partial^2\tilde{g}_{ij}}{\partial x_k \partial x_l}\right)dx_k\otimes dx_l\\
&\ \ \ \ + \ \mbox{lower order terms}.
\end{align*}
Since the $1$--forms $\tilde{\omega}:=\tilde{\phi}^*\omega_N$ when written in a local coodinate system are independent of choice of metric, we obtain
\begin{align*}
-2Rc_{\tilde{g}}+2\theta\tilde{\omega}\otimes\tilde{\omega}-\mathcal{L}_Z\tilde{g}\ \ &= \sum_{i,j,k,l=1}^{m}\tilde{g}^{ij}\dfrac{\partial^2\tilde{g}_{kl}}{\partial x_i \partial x_j}dx_k\otimes dx_l\\
&\ \ \ +\ \mbox{lower order terms}.
\end{align*}
Now for the second equation of the modified Ricci-harmonic-DeTurck flow we have
\begin{align*}
\Delta_{\tilde{g},h}\tilde{\phi}&=\sum_{i,j=1}^{m}\sum_{\gamma=1}^{n}\tilde{g}^{ij}\dfrac{\partial^2\tilde{\phi}_\gamma}{\partial x_i \partial x_j}\partial_\gamma|_{\tilde{\phi}} + \ \mbox{lower order terms},
\end{align*}
and, as
$\mathcal{L}_Z\tilde{\phi}=d\tilde{\phi}(Z)$ has only order terms $1$, it follows
\begin{align*}
\Delta_{\tilde{g},h}\tilde{\phi}-\mathcal{L}_Z\tilde{\phi}=\sum_{i,j=1}^{m}\sum_{\gamma=1}^{n}\tilde{g}^{ij}\dfrac{\partial^2\tilde{\phi}_\gamma}{\partial x_i \partial x_j}\partial_\gamma|_{\tilde{\phi}} + \ \mbox{lower order terms}.
\end{align*}

This way the modified Ricci-harmonic-DeTurck flow  is strictly parabolic. Therefore, the short-time existence and uniqueness of modified Ricci-harmonic-DeTurck flow follows by standard theory of parabolic systems.
\end{proof}

The next lemma is very useful to our intentions, since informs the comportament of tension field in the presence of diffeomorphisms.

\begin{lemma}\label{inv}
Let $\phi:(M^m,g)\to(N^n,h)$ be a smooth map between the Riemannian manifolds $(M,g)$ and $(N,h)$, and $\varphi:M\to M$ a diffeomorphism of $M.$ Then
\begin{equation*}
(\Delta_{\varphi^*(g),h}(\phi\circ\varphi))|_p = \left(\varphi^*\left(\Delta_{g,h}\phi\right)\right)|_p =
(\Delta_{g,h} \phi)|_{\varphi(p)}\in T_{\phi(\varphi(p))}N
\end{equation*}
for all $p\in M$.
\end{lemma}



From now on, for the next results of this section, we will assume that $(M^m,\bar{g})$ and $(N^n,h)$ are two fixed compact Riemannian manifolds and a fixed $1$--form $\omega_N$ in $N$.

The following propositions ensure the correspondence between the solutions of both flows, which will be useful to verify the short-time existence and uniqueness of the modified Ricci-harmonic flow.

\begin{proposition}\label{tFRDeT.tFR}
Suppose the $(\tilde{g}(t),\tilde{\phi}(t))$, $t\in[0,\varepsilon)$, is a solution of Ricci-harmonic-DeTurck flow in $M$. Moreover, let $\varphi_t$, $t\in[0,\varepsilon)$ be a one-parameter family of diffeomorphism of $M$ satisfying
\begin{equation*}
\dfrac{\partial}{\partial t}\varphi_t(p)=Z_t|_{\varphi_t(p)}
\end{equation*}
for all $p\in M$ and all $t\in[0,\varepsilon)$. Then $(g(t),\phi(t))$, $t\in[0,\varepsilon)$, with $g(t):=\varphi^*_t(\tilde{g}(t))$ and $\phi(t):=\varphi^*_t(\tilde{\phi}(t))$, form a solution of modified Ricci-harmonic flow.
\end{proposition}
\begin{proof} First define $\omega(t):=\varphi_t^*(\tilde{\omega}(t)).$ Since holds the identities  $g(t)=\varphi^*_t(\tilde{g}(t))$ and $\phi(t)=\varphi^*_t(\tilde{\phi}(t))$, we obtain
\begin{align*}
\dfrac{\partial}{\partial t}g(t) &= \dfrac{\partial}{\partial t}(\varphi^*_t\tilde{g}(t)) = \varphi^*_t\left(\dfrac{\partial}{\partial t}\tilde{g}(t)+\mathcal{L}_{Z_t}\tilde{g}(t)\right)\\
&=\varphi^*_t\left(-2Rc_{\tilde{g}(t)}+2\theta\tilde{\omega}(t)\otimes\tilde{\omega}(t)\right)\\
&=-2Rc_{g(t)} + 2\theta\omega(t)\otimes\omega(t),
\end{align*}
as well as
\begin{align*}
\dfrac{\partial}{\partial t}\phi(t) &= \dfrac{\partial}{\partial t}(\varphi^*_t\tilde{\phi}(t)) = \varphi^*_t\left(\dfrac{\partial}{\partial t}\tilde{\phi}(t)+\mathcal{L}_{Z_t}\tilde{\phi}(t)\right)\\
&= \varphi^*_t\left(\Delta_{\tilde{g}(t),h}\tilde{\phi}(t)\right) = \Delta_{g(t),h}\phi(t),
\end{align*}
where in the last equality we use the Lemma \ref{inv}.
	
Thus $(g(t),\phi(t))$, $t\in[0,\varepsilon)$, is a solution of modified Ricci-harmonic flow, since $(\tilde{g}(t),\tilde{\phi}(t))$, $t\in[0,\varepsilon)$, is a solution of modified Ricci-harmonic-DeTurck flow.
\end{proof}

\begin{proposition}\label{tFR.tFRDeT}
Suppose the $(g(t),\phi(t))$, $t\in[0,\varepsilon)$ is a solution of modified Ricci-harmonic flow. Moreover, suppose the $\varphi_t$, $t\in[0,\varepsilon)$, is a one-parameter family of diffeomorphisms of $M$ evolving under the map heat flow
\begin{equation*}
\dfrac{\partial}{\partial t}\varphi_t=\Delta_{g(t),\bar{g}}\varphi_t.
\end{equation*}
For each $t\in[0,\varepsilon),$ we define $(\tilde{g}(t),\tilde{\phi}(t))$, a Riemannian metric $\tilde{g}(t)$ and a smooth map $\tilde{\phi}(t)$, by $\varphi^*_t(\tilde{g}(t))=g(t)$ and $\varphi^*_t(\tilde{\phi}(t))=\phi(t)$. Then $(\tilde{g}(t),\tilde{\phi}(t))$, $t\in[0,\varepsilon),$ is a solution of modified Ricci-harmonic-DeTurck flow. Moreover,
\begin{equation}
\dfrac{\partial}{\partial t}\varphi_t(p)=Z_t|_{\varphi_t(p)}
\end{equation}
for all $p\in M$ and all $t\in[0,\varepsilon)$.
\end{proposition}
\begin{proof}
Using Lemma \ref{inv}, we get
\begin{equation*}
\dfrac{\partial}{\partial t}\varphi_t(p)=(\Delta_{g(t),\bar{g}}\varphi_t)|_p =
(\Delta_{\varphi^*_t(\tilde{g}(t)),\bar{g}}\varphi_t)|_p = (\Delta_{\tilde{g}(t),\bar{g}}\mathrm{id}_M)|_{\varphi_t(p)} = Z_t|_{\varphi_t(p)}
\end{equation*}
for all $p\in M$ and all $t\in[0,\varepsilon)$. Now, we define $\tilde{\omega}(t)$ by $\varphi_t^*(\tilde{\omega}(t)):=\omega(t)$. Since $\varphi^*_t(\tilde{g}(t))=g(t)$ e $\varphi^*_t(\tilde{\phi}(t))=\phi(t)$, it follows that
\begin{align*}
\varphi^*_t\left(\dfrac{\partial}{\partial t}\tilde{g}(t)+\mathcal{L}_{Z_t}\tilde{g}(t)\right) &= \dfrac{\partial}{\partial t}g(t) = -2Rc_{g(t)}+2\theta\omega(t)\otimes\omega(t)\\
&= \varphi^*_t\left(-2Rc_{\tilde{g}(t)}+2\theta\tilde{\omega}(t)\otimes\tilde{\omega}(t)\right)
\end{align*}
which implies
\begin{equation*}
\varphi^*_t\left(\dfrac{\partial}{\partial t}\tilde{g}(t)+2Rc_{\tilde{g}(t)}-2\theta\tilde{\omega}(t)\otimes\tilde{\omega}(t)+\mathcal{L}_{Z_t}\tilde{g}(t)\right) = 0,
\end{equation*}
and also
\begin{align*}
\varphi^*_t\left(\dfrac{\partial}{\partial t}\tilde{\phi}(t)+\mathcal{L}_{Z_t}\tilde{\phi}(t)\right) &= \dfrac{\partial}{\partial t}\phi(t) = \Delta_{g(t),h}\phi(t) = \varphi^*_t\left(\Delta_{\tilde{g}(t),h}\tilde{\phi}(t)\right),
\end{align*}
implies
\begin{equation*}
\varphi^*_t\left(\dfrac{\partial}{\partial t}\tilde{\phi}(t)-\Delta_{\tilde{g}(t),h}\tilde{\phi}(t)+\mathcal{L}_{Z_t}\tilde{\phi}(t)\right)=0.
\end{equation*}
Thus, $(\tilde{g}(t),\tilde{\phi}(t))$, $t\in[0,\varepsilon)$, is a solution of modified Ricci-harmonic-DeTurck flow, whenever $(g(t),\phi(t))$, $t\in[0,\varepsilon)$, for a solution of modified Ricci-harmonic flow.
\end{proof}

Finally, we are able to show the short-time existence and uniqueness theorem for modified Ricci-harmonic flow, which is analogous to Hamilton Theorem in \cite{hamilton1} refers to the Ricci flow.

\begin{theorem}[Existence and Uniqueness]\label{1STEUMRF}
There exist a real number $\varepsilon>0$ and one-parameter families of Riemannian metrics $g(t)$ and smooth maps $\phi(t)$, $t\in[0,\varepsilon)$, such that $(g(t),\phi(t))$, $t\in[0,\varepsilon)$, is a solution of modified Ricci-harmonic flow with $g(0)=g$ and $\phi(0)=\phi$. Moreover, the solution $(g(t),\phi(t))$, $t\in[0,\varepsilon)$, is unique.
\end{theorem}
\begin{proof}
Firstly we prove the existence of solution. By Proposition \ref{exs.tFRDet}, there exist a solution $(\tilde{g}(t),\tilde{\phi}(t))$ of modified Ricci-harmonic-DeTurck flow which is defined in some interval $[0,\varepsilon)$ with $\tilde{g}(0)=g$ and $\tilde{\phi}(0)=\phi$. Consequently, we get
\begin{equation*}
\begin{cases}
\dfrac{\partial}{\partial t}\tilde{g}(t)=-2Rc_{\tilde{g}(t)}+2\theta\tilde{\omega}(t)\otimes\tilde{\omega}(t)-\mathcal{L}_{Z_t}\tilde{g}(t)\\
\dfrac{\partial}{\partial t}\tilde{\phi}(t)=\Delta_{\tilde{g}(t),h}\tilde{\phi}(t)-\mathcal{L}_{Z_t}\tilde{\phi}(t),
\end{cases}
\end{equation*}
where $Z_t=\Delta_{\tilde{g}(t),\bar{g}}\mathrm{id}_{M}$. For each $p\in M$, we denote by $\varphi_t(p)$ the solution of the ODE
\begin{equation*}
\dfrac{\partial}{\partial t}\varphi_t(p)=Z_t|_{\varphi_t(p)}
\end{equation*}
with initial condition $\varphi_0(p)=p$. By Proposition \ref{tFRDeT.tFR}, the metrics $g(t)=\varphi^*_t(\tilde{g}(t))$ and maps $\phi(t)=\varphi^*_t(\tilde{\phi}(t))$, $t\in[0,\varepsilon)$, form a solution of the modified Ricci-harmonic flow with $g(0)=g$ and $\phi(0)=\phi$.
	
Now, we prove the uniqueness of our solution. Suppose that
$(g^i(t),\phi^i(t))$, $i=1,2$, are two solutions to the modified Ricci-harmonic flow which are defined on some time interval $[0,\varepsilon)$, satisfying $g^1(0)=g^2(0)$ and $\phi^1(0)=\phi^2(0)$. We want to deduce that  $(g^1(t),\phi^1(t))=(g^2(t),\phi^2(t))$ for all $t\in[0,\varepsilon)$. In order to prove this, we argue by contradiction. Suppose that $(g^1(t),\phi^1(t))\neq(g^2(t),\phi^2(t))$ for some $t\in[0,\varepsilon)$. We define a real number $\delta\in[0,\varepsilon)$ by
\begin{equation*}
\delta=\inf{t\in[0,\varepsilon)}\big\{(g^1(t),\phi^1(t))\neq(g^2(t),\phi^2(t))\big\}.
\end{equation*}
Clearly, $(g^1(\delta),\phi^1(\delta))=(g^2(\delta),\phi^2(\delta))$. Let $\varphi^1_t$ be the solution of map heat flow $\varphi^1_t$
\begin{equation*}
\dfrac{\partial}{\partial t}\varphi^1_t=\Delta_{g^1(t),\bar{g}}\varphi^1_t
\end{equation*}
with initial condition $\varphi^1_\delta=\mathrm{id}_{M}$. Similarly, we denote by $\varphi^2_t$ the solution of map heat flow $\varphi^2_t$
\begin{equation*}
\dfrac{\partial}{\partial t}\varphi^2_t=\Delta_{g^2(t),\bar{g}}\varphi^2_t
\end{equation*}
with initial condition $\varphi^2_\delta=\mathrm{id}_{M}$. It follows from standard parabolic theory that $\varphi^1_t$ and $\varphi^2_t$ are defined on some time interval $[\delta,\delta+\tilde{\varepsilon})$, where $\tilde{\varepsilon}$ is a positive real number. Moreover, if we choose $\tilde{\varepsilon}>0$ small enough, then $\varphi^1_t:M\to M$ and $\varphi^2_t:M\to M$ are diffeomorphism for all $t\in[\delta,\delta+\tilde{\varepsilon})$.
	
For each $t\in[\delta,\delta+\tilde{\varepsilon})$, we define two Riemannian metrics $\tilde{g}^i(t)$ in $M$, $i=1,2$, by $(\varphi^i_t)^*(\tilde{g}^i(t)):=g^i(t)$ and also the smooth maps $\tilde{\phi}^i(t)$ of $M$ in $(N,h)$, $i=1,2$, by $(\varphi^i_t)^*(\tilde{\phi}^i(t)):=\phi^i(t)$. It follows from Proposition \ref{tFR.tFRDeT} that $(\tilde{g}^1(t),\tilde{\phi}^1(t))$ and $(\tilde{g}^2(t),\tilde{\phi}^2(t))$ are two solutions of the modified Ricci-harmonic-DeTurck flow. Since $(\tilde{g}^1(\delta),\tilde{\phi}^1(\delta))=(\tilde{g}^2(\delta),\tilde{\phi}^2(\delta))$, the uniqueness statement in Proposition \ref{exs.tFRDet} implies that $(\tilde{g}^1(t),\tilde{\phi}^1(t))=(\tilde{g}^2(t),\tilde{\phi}^2(t))$ for all $t\in[\delta,\delta+\tilde{\varepsilon})$. For each  $t\in[\delta,\delta+\tilde{\varepsilon})$, we define a vector field $Z_t$ in $M$ by
\begin{equation*}
Z_t=\Delta_{\tilde{g}^1(t),\bar{g}}\mathrm{id}_{M}=\Delta_{\tilde{g}^2(t),\bar{g}}\mathrm{id}_{M}.
\end{equation*}
Finally, by Proposition \ref{tFR.tFRDeT}, we have
\begin{equation*}
\dfrac{\partial}{\partial t}\varphi^1_t(p)=Z_t|_{\varphi^1_t(p)}
\end{equation*}
and
\begin{equation*}
\dfrac{\partial}{\partial t}\varphi^2_t(p)=Z_t|_{\varphi^2_t(p)}
\end{equation*}
for all $p\in M$ and all $t\in[\delta,\delta+\tilde{\varepsilon})$. Since $\varphi^1_\delta=\varphi^2_\delta=\mathrm{id}_{M}$, it follows that $\varphi^1_t=\varphi^2_t$ for all $t\in[\delta,\delta+\tilde{\varepsilon})$. Putting these facts together, we conclude that
\begin{equation*}
g^1(t)=(\varphi^1_t)^*(\tilde{g}^1(t))=(\varphi^2_t)^*(\tilde{g}^2(t))=g^2(t)
\end{equation*}
and
\begin{equation*}
\phi^1(t)=(\varphi^1_t)^*(\tilde{\phi}^1(t))=(\varphi^2_t)^*(\tilde{\phi}^2(t))=\phi^2(t).
\end{equation*}
This  contradicts the definition of $\delta$.
\end{proof}

\subsection{Modified Ricci-harmonic soliton}

In this section we characterized a modified Ricci soliton as part of a solution of modified Ricci-harmonic flow, being this, the second motivation for the introduction of modified Ricci solitons. In addition, it represents a new characterization for the $m$--quasi-Einstein metrics.

We will admit here that the solution exists and is unique in short time (fact proven in the previous section for the case where $M^m$ and $N^n$ are compact).  In this way the modified Ricci solitons are incorporated in the definition of modified Ricci-harmonic soliton (cf. Definition \ref{Msoliton}) and, by Proposition \ref{equiv}, we deduce that this last concept is equivalent to a special solution of the modified Ricci-harmonic flow (cf. Definition \ref{spc}).

\begin{definition}\label{Msoliton}
A modified Ricci-harmonic soliton $(M^m,g,\phi)$ is a  modified Ricci soliton $(M^m,g,X,\omega)$ provided with a smooth map $\phi:(M^m,g)\to (N^n,h)$, for some fixed Riemannian manifold $(N,h)$ and a fixed $1$--form $\omega_N$ on $N$ satisfying
\begin{equation}\label{phi0}
\begin{cases}
\phi^*\omega_N=\omega,\\
\Delta_{g,h}\phi=\mathcal{L}_X\phi.
\end{cases}
\end{equation}
\end{definition}

We know that a special solution (self-similar solution) $g(t)$, $t\in[0,\varepsilon)$, of Ricci flow is invariant for scaling of positive constants $c(t)$ and diffeomorphisms $\psi_t$ of $M$ which vary smoothly for one temporal parameter $t$ of an initial metric $g$, i.e., $g(t):=c(t)\psi_t^*g$. Therefore, it is reasonable that the family of maps $\phi(t):M\to N$ be defined by $\phi(t):=\psi^*_t\phi$, for an initial map $\phi:(M,g)\to(N,h)$.

\begin{definition}\label{spc}
We say that $(g(t),\phi(t))$, $t\in[0,\varepsilon)$, is a special solution of modified Ricci-harmonic flow if there exist a one-parameter family of diffeomorphism $\psi_t$ of $M^n$ satisfying $\psi_0=\mathrm{id}_M$ and positive scalars $c(t)$ varying smoothly in $t$, such that
\begin{equation}\label{selfsim}
\begin{cases}
g(t)=c(t)\psi^*_tg,\\
\phi(t)=\psi^*_t\phi=\phi\circ\psi_t.
\end{cases}
\end{equation}
\end{definition}

Inspired by Lemma $2.4$ in \cite{TRF:AI}, referring to the Ricci solitons, we similarly obtain the relation between a modified Ricci-harmonic soliton and a special solution of the modified Ricci-harmonic flow.


\begin{proposition}\label{equiv}
Let $(g(t),\phi(t))$, $t\in[0,\varepsilon)$, be a special solution of the modified Ricci-harmonic flow, then there exist a smooth map $\phi:(M^m,g)\to(N^n,h)$ satisfying \eqref{phi0} for a fixed $1$--form and $\omega_N$ on $N$ and a vector field $X$ on $M$, such that $(M^m,g,X,\omega)$ is a modified Ricci soliton, i.e., $(M^m,g,\phi)$ is a modified Ricci-harmonic soliton. Conversely, given a modified Ricci-harmonic soliton $(M^m,g,\phi)$  there exist a one-parameter family of diffeomorphism $\psi_t$ of $M$ and positive scalars $c(t)$ varying smoothly in $t$ such that $(g(t),\phi(t))$, $t\in[0,\varepsilon)$, with $g(t):=c(t)\psi^*_tg$ and $\phi(t):=\psi^*_t\phi$, is a special solution of modified Ricci-harmonic flow.
\end{proposition}
\begin{proof}
Initially suppose that $(g(t),\phi(t))$, $t\in[0,\varepsilon)$, is a special solution of modified Ricci-harmonic flow. Since $\psi_t$ is a family of diffeomorphisms of $M$ satisfying $\psi_0=\mathrm{id}_M$ and $c(t)$ are positive scalars varying smoothly in $t$, we may assume without loss generality that $c(0)=1$. Define $\omega$ on $M$ by $\omega:=\phi^*\omega_N$ for fixed $1$--form $\omega_N$ on $N$. So, by definition of $\omega(t)$ we obtain
\begin{equation*}
\omega(t)=\phi(t)^*\omega_N=(\phi\circ\psi_t)^*\omega_N=\psi_t^*\phi^*\omega_N=\psi_t^*\omega,
\end{equation*}
thus
\begin{equation*}
-2Rc_{g}+2\theta\omega\otimes\omega=\dfrac{\partial}{\partial t}g(t)\Big|_{t=0}=c'(0)g+\mathcal{L}_{Y_0}g,
\end{equation*}
where $Y_t$ is the family of vector fields generating by the diffeomorphisms $\psi_t$. Thus $(M,g,X,\omega)$ is a modified Ricci soliton with $\lambda=-\dfrac{c'(0)}{2}$ and $X=Y_0$. Moreover,
\begin{align*}
\Delta_{g,h}\phi  &= \Delta_{g(t),h}\phi(t)\Big|_{t=0}
=\dfrac{\partial}{\partial t}\phi(t)\Big|_{t=0}\\
&=\dfrac{\partial}{\partial t}\psi^*_t\phi\Big|_{t=0}
=\psi^*_t\big(\mathcal{L}_{Y_t}\phi\big)\Big|_{t=0}=\mathcal{L}_{X}\phi.
\end{align*}
Hence, $(M,g,\phi)$ is a modified Ricci-harmonic soliton.
	
Conversely, suppose that $(M,g,\phi)$ is a modified Ricci-harmonic soliton. Define $c(t)=1-2\lambda t,$ and a one-parameter family of vector fields $Y_t$ on $M$ by $Y_t=\dfrac{X}{c(t)}$.
	
Let $\psi_t$ the diffeomorphisms generating by the family $Y_t$, i.e., $Y_t|_{\psi_t(p)}=\dfrac{\partial}{\partial t}\psi_t(p)$, where $p\in M$ and $\psi_0=\mathrm{id}_{M}$. So we can define the one-parameter families of Riemannian metrics on $M$ and smooth maps from $M$ to $N$ by
\begin{equation*}
g(t):=c(t)\psi_t^*g \ \ \mbox{and}\ \ \phi(t):=\psi_t^*\phi,
\end{equation*}
respectively. Since $\phi^*\omega_N=\omega$ one has
\begin{equation*}
\psi^*_t\omega=\psi^*_t(\phi^*\omega_N)=(\phi\circ\psi_t)^*\omega_N=\phi(t)^*\omega_N,
\end{equation*}
so we define $\omega(t):=\phi(t)^*\omega_N$. Hence we have that $(g(t),\phi(t))$, $t\in[0,\varepsilon)$, is a special solution of modified Ricci-harmonic flow, because
\begin{align*}
\dfrac{\partial}{\partial t}g(t)
&=\psi^*_t(-2\lambda g+\mathcal{L}_Xg)\\
&=\psi^*_t(-2Rc_{g}+2\theta\omega\otimes\omega)\\
&=-2Rc_{g(t)}+2\theta\omega(t)\otimes\omega(t),
\end{align*}
as well as
\begin{align*}
\dfrac{\partial}{\partial t}\phi(t)&=\psi^*_t\big(\mathcal{L}_{Y_t}\phi\big)
=\dfrac{1}{c(t)}\psi^*_t\big(\mathcal{L}_X\phi\big)\\
&=\dfrac{1}{c(t)}\psi^*_t\big(\Delta_{g,h}\phi\big)
=\Delta_{g(t),h}\phi(t).
\end{align*}
\end{proof}

\section{A modified Ricci-harmonic soliton}\label{HR-example}

In this section we take $m\geqslant2,$ $n\geqslant1$ integers. Let $x=(x_{1},\ldots,x_{m})\in\mathbb{R}^m$ and let $y=(y_{1},\ldots,y_{n})\in\mathbb{R}^n$ be standard coordinate systems. 

In this section, let $(M^m,g)$ be warped product $\mathbb{R}\times_{e^{x_1}}\mathbb{R}^{m-1}$ and let  $(N^n,h)$ be the punctured Euclidean space $(\mathbb{R}^n-\{\vec{a}\},g_{\mathbb{R}^n}),$ where $\vec{a}=(a_1,\ldots,a_n)$ with $a_{\gamma}>0$ for $\gamma=1,\ldots,n.$ 

For $\lambda<1-m<0$ by taking $\theta=\dfrac{1}{1-\lambda-m}>0,$ define
\begin{itemize}
\item $X\in\mathfrak{X}(M):$ 
\begin{align*}
X_{x} = \big(m-\lambda-1,2\lambda x_2,\ldots,2\lambda x_m\big);    
\end{align*}
\item $\omega\in\Omega^1(M):$ 
\begin{align*}
\omega = \dfrac{1}{\theta}dx_1;    
\end{align*}
\item $\omega_N\in\Omega^1(N):$
\begin{align*}
\omega_N=\dfrac{-1}{n\theta \lambda}\sum_{\gamma=1}^n\dfrac{dy_\gamma}{y_\gamma-a_\gamma};
\end{align*}
\item $\phi\in C^{\infty}(M;N):$
\begin{align*}
\phi(x_1,\ldots,x_m) = e^{-\lambda x_1}\ \vec{b}+\vec{a}, \end{align*}
where $\vec{b}=(b_1,\ldots,b_n)$ with $b_{\gamma}>0$ for $\gamma=1,\ldots,n.$
\end{itemize}
 
Hence, $(M^m,g,\phi)$ is modified Ricci-harmonic soliton, i.e., holds
\begin{align*}
\begin{cases}
\phi^*\omega_N=\omega,\\
Rc_{g}+\dfrac{1}{2}\mathcal{L}_Xg=\lambda g + \theta\omega\otimes\omega,\\
\Delta_{g,h}\phi = \mathcal{L}_X\phi.
\end{cases}    
\end{align*}

In fact,
\begin{itemize}
\item $\phi^*\omega_N = \omega:$ 
\begin{align*}
\phi^*\omega_N &= \dfrac{-1}{n\theta \lambda}\sum_{\gamma=1}^n\dfrac{d\phi_\gamma}{\phi_\gamma-a_\gamma} = \dfrac{-1}{n\theta \lambda}\sum_{\gamma=1}^n \dfrac{(-\lambda)(\phi_\gamma-a_\gamma)}{\phi_\gamma-a_\gamma}dx_1\\
&=\dfrac{1}{n\theta}\sum_{\gamma=1}^n dx_1 = \dfrac{1}{n\theta} n dx_1 = \dfrac{1}{\theta} dx_1 = \omega;
\end{align*}

\item $Rc_{g} + \dfrac{1}{2}\mathcal{L}_Xg = \lambda g +\theta \omega\otimes\omega:$
\begin{align*}
Rc_{g} +\dfrac{1}{2}\mathcal{L}_Xg &= -(m-1)dx_1^2 -(m-1)e^{2x_1}\sum_{i=2}^m dx_i^2\\
& \ \ \ +(m-\lambda-1)e^{2x_1}\sum_{i=2}^m dx_i^2\\
&= -(m-1)dx_1^2 + \lambda e^{2x_1}\sum_{i=2}^m dx_i^2 +\lambda dx_1^2 -\lambda dx_1^2\\
&= \lambda dx_1^2 + \lambda e^{2x_1}\sum_{i=2}^m dx_i^2 - (m+\lambda-1)dx_1^2\\
& = \lambda dx_1^2 + \lambda e^{2x_1}\sum_{i=2}^m dx_i^2 + \dfrac{1}{\theta}dx_1^2 = \lambda g + \theta\omega\otimes\omega;
\end{align*}

\item $\Delta_{g,h}\phi=\mathcal{L}_X\phi:$
\begin{align*}
\Delta_{g,h}\phi &= \Big(\Delta_g\phi_1,\ldots,\Delta_g\phi_n\Big)\\ 
&= \left(\dfrac{d^2\phi_1}{d x_1^2}+(m-1)\dfrac{d \phi_1}{d x_1},\ldots,\dfrac{d^2\phi_n}{d x_1^2}+(m-1)\dfrac{d \phi_n}{dx_1}\right)\\
&= (m-\lambda-1)\left(\dfrac{d\phi_1}{d x_1},\ldots,\dfrac{d\phi_n}{d x_1}\right)= \Big(g(\nabla\phi_1,X),\ldots,g(\nabla\phi_n,X)\Big)\\
&= \Big(d\phi_1(X),\ldots,d\phi_n(X)\Big) = d\phi(X) = \mathcal{L}_X\phi.
\end{align*}
\end{itemize}

Finally, we will get a special solution $(g(t),\phi(t)),$ $t\in(1/2\lambda,+\infty),$ of modified Ricci-harmonic flow. To do this, we must find a one-parameter family of diffeomorphism $\psi_t=(\psi_t^1,\ldots,\psi_t^m):\mathbb{R}^m\to\mathbb{R}^m,$ that is the solution to the following PDE
\begin{align*}
\begin{cases}
\dfrac{\partial}{\partial t}\psi_t(x) = \dfrac{1}{1-2\lambda t}X_{\psi_t(x)}\\
\psi_0(x)=x,
\end{cases}
\end{align*}
which is equivalent to
\begin{align*}
\begin{cases}
\dfrac{\partial}{\partial t}\psi_t^1(x) = \dfrac{m-\lambda-1}{1-2\lambda t},\\
\psi_0^1(x)=x_1, 
\end{cases}
\ \ \mbox{and} \ \ \ 
\begin{cases}
\dfrac{\partial}{\partial t}\psi_t^i(x) = \dfrac{2\lambda x_i}{1-2\lambda t},\\
\psi_0^i(x)=x_i, \ \ \ i=2,\ldots,m,
\end{cases}
\end{align*}
whose solutions are
\begin{align*}
\psi_t^1(x)=\ln(1-2\lambda t)^{(\lambda-m+1)/2\lambda} + x_1
\end{align*}
and
\begin{align*}
\psi_t^i(x) = \big(1-\ln(1-2\lambda t)\big)x_i, \ \ \ i=2,\ldots,m.   
\end{align*}

Therefore, for $t\in(1/2\lambda,+\infty),$ 
\begin{align*}
g(t)&=c(t)dx_1^2 + c(t)^{(2\lambda-m+1)/\lambda} \cdot \big(1-\ln(1-2\lambda t)\big)^2 e^{2x_1}\sum\limits_{i=2}^m dx_i^2
\end{align*}
and
\begin{align*}
\phi(t)(x)=c(t)^{(m-\lambda-1)/2} e^{-\lambda x_1}\vec{b}+\vec{a},
\end{align*}
where $c(t)=1-2\lambda t,$ is a special solution of modified Ricci-harmonic flow.   

\begin{remark}
Recall that the gradienticity is a Riemannian metric condition. In fact, for the standard metric of $\mathbb{R}^{m},$
\begin{align*}
g_{\mathbb{R}^m} = \sum\limits_{i=1}^{m}dx_{i}^{2},   
\end{align*}
the vector field 
\begin{align*}
X_{x}=(m-\lambda-1,2\lambda x_{2},\ldots,2\lambda x_{m}), 
\end{align*}
is the gradient of smooth function $u:\mathbb{R}^{m}\to\mathbb{R}$ given by
\begin{align*}
u(x_{1},\ldots,x_{m}) = (m-\lambda-1)x_1+\lambda\|(x_{2},\ldots,x_{m})\|^{2} + \kappa,    
\end{align*}
for some constant $\kappa.$    

On the other hand, 
\begin{align*}
X^{\flat} &= g(X,\cdot) = dx_1(X)dx_1 + e^{2x_1}\sum\limits_{i=2}^{m}dx_{i}(X)dx_{i}\\
&= X_{1}dx_1 + e^{2x_1}\sum\limits_{i=2}^{m} X_{i}dx_{i} = (m-\lambda-1)dx_1 + 2\lambda e^{2x_1}\sum\limits_{i=2}^{m}x_{i}dx_{i},
\end{align*}
which implies, by exterior derivative
\begin{align*}
dX^{\flat} = 4\lambda e^{2x_1}\sum\limits_{i=2}^{m}x_{i}dx_1\wedge dx_{i} \neq 0,   
\end{align*}
i.e., $X$ is nongradient smooth vector field in the Riemannian metric $g.$
\end{remark}


\vspace{0.3cm}
\textbf{Acknowledgements:} The author would like to express his gratitude to the anonymous reviewer for his careful reading and for his corrections and constructive suggestions to the text. Thanks to your work, we were able to improve this article.

\vspace{0.3cm}
\textbf{Data availability:} No data was used for the research described in the article.

\vspace{0.8cm}
{\huge\textbf{Declarations}}

\vspace{0.6cm}
\textbf{Conflict of interest:} The author declares that he doesn't have any financial and non-financial conflict of interests.

\end{document}